\documentclass[]{amsart}

\usepackage{latexsym}
\usepackage{amssymb,amsmath,amsopn}
\usepackage[dvips]{graphicx}   
\usepackage{wrapfig}
\usepackage{subfig}
\usepackage{color,epsfig}      
\usepackage{url}
\usepackage{amsbsy}

\newtheorem{theorem}{Theorem}[]

\newtheorem{statement}{Statement}[]

\newtheorem{defi}{Definition}[]

\theoremstyle{definition}
\newtheorem{corollary}{Corollary}[]

\newtheorem{remark}[]{Remark}[]

\begin{document}
\title[]{Diameter, width and thickness in the hyperbolic plane}
\author[\'A. G.Horv\'ath]{\'Akos G.Horv\'ath}
\address {\'A. G.Horv\'ath \\ Department of Geometry \\ Mathematical Institute \\
Budapest University of Technology and Economics\\
H-1521 Budapest\\
Hungary}
\email{ghorvath@math.bme.hu}


\begin{abstract}
This paper contains a new concept to measure the width and thickness of a convex body in the hyperbolic plane. We compare the known concepts with the new one and prove some results on bodies of constant width, constant diameter and given thickness.
\end{abstract}

\date{}

\maketitle

\section{Introduction}

In hyperbolic geometry there are several concepts to measure the breadth of a convex set. In the first section of this paper we introduce a new idea and compare it with four known ones. Correspondingly, we define three classes of bodies, bodies of constant with, bodies of constant diameter and bodies having the constant shadow property, respectively. In Euclidean space more or less these classes are agree but in the hyperbolic plane we need to differentiate them. Among others we find convex compact bodies of constant width which size essential in the fulfilment of this property (see Statement \ref{st:regularpoly}) and others where the size dependence is non-essential (see Statement{st:circle}). We prove that the property of constant diameter follows to the fulfilment of constant shadow property (see Theorem \ref{thm:constdiamandconstshadow}), and both of them are stronger as the property of constant width (see Theorem \ref{thm:constdiam}). In the last part of this paper, we introduce the thickness of a constant body and prove a variant of Blaschke's theorem on the larger circle inscribed to a plane-convex body of given thickness.

\subsection{The hyperbolic concepts of width (breadth) introduced earlier}

\subsubsection{$\mathrm{width}_1(\cdot)$:}
Santal\'o in his paper \cite{santalo} developed the following approach. For a convex set $K$, let $z$ be a point in the boundary of $K$. Let $L_z$
be the supporting line of $K$ passing through $z$ and let $L_z^\bot$ be the orthogonal line to $L_z$ at the point $z$. Choose another supporting line $L_v$ of $K$  that is orthogonal to $L_z^\bot$, then the \emph{breadth of $K$ corresponding to $z$} is the hyperbolic distance between $L_v$ and $L_z$, i.e., the length of the segment $L_z^\bot$ joining the lines $L_v$ and $L_z$. Under this approach, the set $K$ is said to have constant width $\lambda$ if the breadth of $K$ with respect to every point in its boundary is $\lambda$. The problem with this approach that there is no chance to define "opposite points" on the boundary of $K$, and the Euclidean concept of affine diameter cannot be interpretable.

\subsubsection{$\mathrm{width}_2(\cdot)$:}
Historically the next development is due to  J. Filmore in \cite{filmore} who introduced the following notion. If a set of the hyperbolic plane holds the property that with every pair $x,y$ of its points it contains the intersection of all paracyclic (horocyclic) domains which contain $x$ and $y$, then we say that the set is a h-convex set of the plane. Of course every h-convex set is also convex in the standard hyperbolic meaning but for example a convex polygon is not h-convex. Let $K$ be an h-convex set with smooth boundary and fix a point $O$ in the plane. Consider a half-line through the point $O$ and consider that pair of paracycles which tangents to $K$ and go through on the ideal point of the given half-line. The width in the direction of this half-line is the length of the segment which endpoints are the intersection points of the above paracycles with the line containing the given half-line. The width in this approach depends on the point $O$ and cannot measure the sets which convex but not h-convex one.

\subsubsection{$\mathrm{width}_3(\cdot)$:}
K. Leichtweiss introduced a concept in \cite{leichtweiss} which don't use touching lines or paracycles to measure the width of a h-convex set. His idea is to start with strip regions which are homogeneous in the sense that they admit one-parameter subgroups of motions, which in the hyperbolic case is the group of translations correspondingly to a given line. (More precisely, regard those translations which are the product of two reflections in lines which are orthogonal to the given line. These translations form a subgroup of motions of the hyperbolic plane. The strips which invariant against the elements of this subgroup must be bounded by geodesic parallel curves of constant curvature, which are the so-called hypercycles (distance curves) of the given line.) Thus it is quite natural to define the width a horocyclic convex curve $C$ in a certain direction by means of the width of the supporting strip region of this direction. In this approach the condition of h-convexity also needs because the curvature of the boundary points of the strip region is less than the curvature of paracycles but again it is non-zero. Consequently, this method cannot be applied for polygons, too. On the other hand it is a strong connection with the notions of spherical geometry and we can define appropriate signed support function $h(\theta)$ for a h-convex set like in the classical convex geometry. The width $w(\theta)$ in a direction holds the equality $w(\theta)=h(\theta)+h(\theta+\pi)$. Observe that these functions depends on the origin and so the "width of the h-convex set in a direction" doesn't an inner property of the set, like the earlier concept of Filmore. Advantage of this building up that in the case of h-convex sets we can say about "opposite points" and "affine diameters" as in the Euclidean situation.

\subsubsection{$\mathrm{width}_4(\cdot)$:}
Recent approach in the literature is in the paper of J. Jeronimo-Castro and F. G. Jimenez-Lopez \cite{jeronimo-castropaper}. The authors first introduce  the notion of opposite points for an arbitrary compact convex body $K$ of the hyperbolic plane. We translate the definition to the model-free language of the hyperbolic plane. Transform with an isometry the point $z$ of the boundary of $K$ to a given (and fixed) point $O$ of the plane such a way that a given (and fixed) line $l$ of the plane through $O$ will be a supporting (tangent) line $L_z$ of the image. In the second step search a point $w$ on the boundary of $K$ for which a support line $L_w$ at $w$ has equal alternate angle with the segment $[z,w]$ as the angle between $[z,w]$ and $L_z$. By definition such a pair $z$ and $w$ of points are \emph{opposite points} with respect to $K$ and the segment $[z,w]$ is an \emph{isometric diameter}. The existence at least one opposite pair to a given $z$ can be seen immediately from the Poincare's disk model. Similarly, it can be seen easily that if $K$ smooth and h-convex $z$ uniquely determines its opposite pair $w$. Let denote by $L_{z,w}$ the line of the segment $[z,w]$. Since $L_{z,w}$ has equal alternate angles with $L_z$ and $L_w$ these latter are ultraparallel lines (non intersected and non-parallel) implying that, there is an unique distance of it. This distance is the width of the body corresponding to the isometric diameter $[z,w]$. Note that in this approach there is no such width function which for every direction uniquely determines a real number. It is possible, that for the same boundary point $z$ we find two opposite points $w'$ and $w''$ for which the corresponding width are distinct to each other. But immediately can be defined the notion of \emph{curve of constant width} with the property that all widths let be the same value with respect to all pairs of opposite points.

Our method gives a synthesization of the above concepts preserving the important properties of those ones.

\subsection{Comparison of the known concepts of width}

First we observe that we can compare the Santal\'o's breadth $\mathrm{width}_1(z)$ with $\mathrm{width}_4(z)$ of Jeronimo-Castro and Jimenez-Lopez. Both of them corresponds the value of width to the points of the boundary of $K$. Concentrating to the method of \cite{jeronimo-castropaper}, if the boundary of $K$ smooth and horocyclic convex, a point on the boundary has an unique opposite pair and the width with respect to this point can be defined as the distance of the opposite (in this sence) tangent lines.

\begin{statement}\label{st:santalo_jeronimo-castro}
Let $K$ be a convex compact body of the plane with smooth and h-convex boundary. Then we have
\begin{equation}\label{ineq:santalo_jeronimo-castro}
\mathrm{width}_{4}(z)\geq \mathrm{width}_1(z).
\end{equation}
Equality holds at that point $z$ where the normal is a double-normal.
\end{statement}

\begin{proof}
The breadth of \cite{santalo} can be seen as the distance of the tangent at the points $z$ and $v$, while the width of \cite{jeronimo-castropaper} can be found as the distance of the lines $L_z$ and $L_w$ (see Fig. \ref{fig:santalo_jeronimo-castro}). From synthetic absolute geometry (see in \cite{appendix}) we know that this latter quantity attend at the perpendicular to $L_z$ from the midpoint $M$ of the segment $[z,w]$.

\begin{figure}[ht]
\includegraphics[scale=0.6]{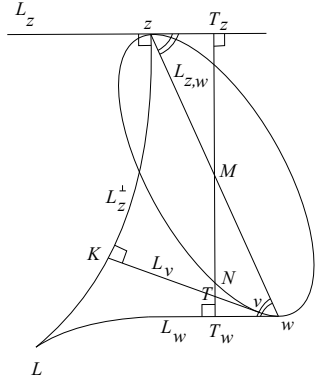}
\caption{Comparison of Santal\'o width function to the width of Jer\'onimo-Castro and Jimenez-Lopez}\label{fig:santalo_jeronimo-castro}
\end{figure}

In the figure it is the length of the segment $[T_z,T_w]$. Since the points $w$ and $z$ are in distinct half-planes of the line $T_zT_w$. the boundary of $K$ intersects the segment $[T_zT_w]$ in two points, one of them (say $N$) is on the segment $[M,T_w]$. Note that the lines $L_w$ and $L_z^\bot$ either non-intersecting or intersect each other in a point $L$. In the second case, the quadrangle with vertices $L$,$z$,$T_z$ and $T_w$ is a Saccheri quadrangle implying that $T_wLz_{\angle}$ is acute angle. From this follows that the touching point $v$ -- which needed to the determination of the Santalo's breadth -- is on the arc $wNz$. This also true in the first case, hence $L_v$ intersects the segment $[N,T_w]$ in a point $T$. This provides the following inequalities:
\begin{equation}
\mathrm{width}_{4}(z)=|[T_z,T_w]|\geq |[T,T_z]|\geq |[K,z]|=\mathrm{dist}(L_z,L_v)=\mathrm{width}_1(z).
\end{equation}
(The second inequality follows from the fact that the distance of two lines is less or equal to the length of another transversalis of them.) Clearly in the case of equality we have $z=T_z$, $w=v=T_w$ and the isometric diameter $[z,w]$ is a double normal of the boundary of $K$.
\end{proof}

Our second note that there is a simple comparison of $\mathrm{width}_2$ and $\mathrm{width}_3$. As we can see on Fig. \ref{fig:fillmore_leichtweiss} these functions correspond to a direction of the plane defined by a half-line starting from a given point $O$ of the plane. Here we assume that the origin $O$ is in the interior of the body $K$.  (We can consider that the origin is an arbitrary point of the plane, but more comfortable if the getting support functions have only positive values, respectively. In the case of Fillmore's paper this follows from the fact that the domain of the investigated paracycles contains $K$ and Leichtweiss directly assumed that the origin is in the interior of $K$.)
\begin{statement}\label{st:fillmore_leichtweiss}
Let $K$ be a horocycle-convex compact body of the plane with the origin in its interior. Then we have
\begin{equation}\label{ineq:fillmore_leichtweiss}
\mathrm{width}_{2}(\theta)\geq \mathrm{width}_3(\theta).
\end{equation}
Moreover, equality holds if and only if the touching points $z$ and $v$ are the endpoints of a double-normal of the boundary of $K$.
\end{statement}

\begin{figure}[ht]
\includegraphics[scale=0.6]{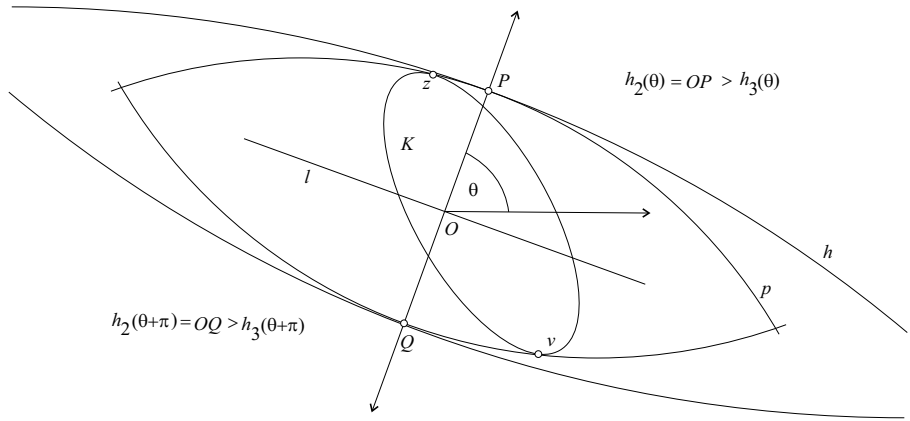}
\caption{Connection between the width function of Fillmore and the width function of Leichtweiss}\label{fig:fillmore_leichtweiss}
\end{figure}

\begin{proof}
 Let denote by $h_2(\cdot)$ and $h_3(\cdot)$ the corresponding support functions. If $\theta$ any parameter determining the direction of the half-line $\overrightarrow{OP}$ then  $\mathrm{width}_2(\theta)=h_2(\theta)+h_2(\theta+\pi)=|OP|+|OQ|$. In Fig.\ref{fig:fillmore_leichtweiss} $p$ is that paracycle which touches $K$. One of its axis is $\overrightarrow{OP}$. The angle $\theta$ also corresponds to a value of the support function $h_3(\theta)$ of paper \cite{leichtweiss}. Since it should be the distance of a hypercycle of the leading line $l$, which is perpendicular to the direction $\overrightarrow{OP}$ we have to consider the hypercycle $h$ which goes through the point $P$. However, $h$ doesn't touch $K$, because its curvature is less than the curvature of the paracycle $p$. This means that $h_3(\theta)\leq h_2(\theta)$ consequently $\mathrm{width}_3(\theta)\leq \mathrm{width}_2(\theta)$ as we stated. Equality holds for a $\theta$ if and only if $P=z$ and $Q=v$ meaning that the normals at the touching points go through the origin $O$, respectively. Hence the segment $[z,v]$ is a double-normal of the boundary of $K$ (and so it is also an isometric-diameter).
\end{proof}

\section{A new width function of the hyperbolic plane}

In this section, we propose a new definition of the width function. We eliminate that dependence of the width function on the origin, which is present in the concept $\mathrm{width}_3$. Using again hypercycles we define such strips which leading-lines don't go through a fix origin. Denote by $\mathcal{I}$ the set of ideal points. (To name "ideal point" in our terminology means a point on the boundary of a disk model. In a model-independent approach an ideal points is such a new object which corresponds to a congruence class of hyperbolic parallel half-lines.)

\begin{defi}
Let $O$ be a point in the interior of the convex compact body $K$. Let $\overrightarrow{OX}$ be a half-line with ideal point $X$ and consider a line $YX$ where $Y\in \mathcal{I}$, $YX$ is parallel to $\overrightarrow{OX}$ and intersects $K$. Then there are two hypercycles corresponding to the leading-line $YX$, which touch $K$ and the strip between them contains $K$. The distance of these distance-curves are $d_K^+(YX)$ and $d_K^-(YX)$, respectively. Let denote by $d_K(YX):=d_K^+(YX)+d_K^-(YX)$. The \emph{width of $K$ in the direction of $X$} is
\begin{equation}\label{eq:width}
\mathrm{width}_K(X):=\sup\{d_K(YX) : Y\in \mathcal{I} \}.
\end{equation}
\end{defi}

It follows by a standard compactness argument, that for an ideal point $X$ there is at least one $Y^\star$ such that $\mathrm{width}_K(X)=d_K(Y^\star X)$. So the supremum in (\ref{eq:width}) is a maximum.

\begin{remark}
In Euclidean geometry the width in a direction means the needed size of a gate on which we can push across the body along the given direction (without any rotation). In hyperbolic geometry, we imagine this gate at an ideal point. We can push the body along several distinct half-lines corresponding to this point. The gate has to be enough big to that, the passage through the gate should be successful regardless of which route we used. This explains why we chose supremum instead of infimum in our definition.
\end{remark}

\begin{statement}\label{st:lecthweissantalo}
For every ideal point $X$ holds the inequality: $\mathrm{width}_K(X)\geq \mathrm{width}_3(X)$. Moreover, for every $X\in \mathcal{I}$ we can chose a point $z(X)$ on the boundary of $K$ for which holds the inequality $\mathrm{width}_K(X)\geq \mathrm{width}_1(z(X))$. If $K$ is strictly convex and smooth then there is a homeomorphism $\varphi:\mathcal{I}\longrightarrow \mathrm{bd}(K)$ $\varphi(X)=z(X)$ such that $\mathrm{width}_K(X)\geq \mathrm{width}_1(\varphi(X))$.
\end{statement}

\begin{proof}
It is clear, that if $O'$ any point of the interior of $K$ then all half-line $\overrightarrow{O'X}$ is a subset of such a line $YX$ which is parallel to $\overrightarrow{OX}$. Hence  $\mathrm{width}_K(X)\geq \mathrm{width}_3(\overrightarrow{OX})$ holds with respect to any pair of $O$ and $\overrightarrow{OX}$.

\begin{figure}[ht]
\includegraphics[scale=0.6]{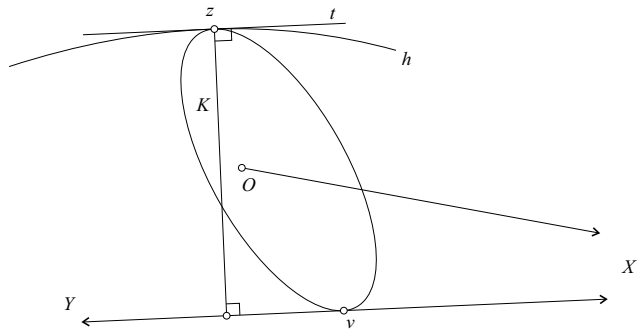}
\caption{Comparisation of the width function of Santal\'o with $\mathrm{width}_K(\cdot)$}\label{fig:santalo_width}
\end{figure}

For a given direction $\overrightarrow{OX}$ let $Y$ such that $YX$ touches $K$ at the point $v$. One of the corresponding hypercycles is the line $YX$ and the other touches the boundary of $K$ at the point $z$. Now the distance of the point $z$ to $YX$ is $d_K(YX)$ which also represents $\mathrm{width}_1(z)$ (see Fig.\ref{fig:santalo_width}). Hence if we correspond to the direction $\overrightarrow{OX}$ the above point $z(X)$ on the boundary of $K$ we get that $\mathrm{width}_1(z(X))\leq \mathrm{width}_K(X)$.

To prove the last statement consider a Cayley-Klein model. In this model the smooth and strictly convex body has the same property. Obviously we can assume that $K$ contains the origin $O$ of the model. For a point $X$ of the model circle we can associate two points, which are the tangent points of the tangent lines of $K$ through $X$. We start from that point $X_0$ of the model circle for which $\overrightarrow{OX_0)$ is the positive half of the axis $x$ of an Euclidean coordinate system. Let $\varphi(X_0}$ that tangent point for which the line $\varphi(X_0)X_0$ has positive directional tangent. Hence the triangle $O\varphi(X_0){X_0}_{\triangle}$ has positive (counterclockwise) direction and by definition we choose that tangent point as the $\varphi(X)$ image of $X$ for which this property holds. Clearly our conditions on the boundary of $K$ guarantees that $\varphi$ is a homeomorphism.

\end{proof}

We have also an immediate connection between $\mathrm{width}_4(z)$ and our function $\mathrm{width}_K(X)$.

\begin{statement}\label{st:isometricdiam}
Let $K$ be a h-convex compact body. For every point $X\in\mathcal{I}$ there is a point $Y\in \mathcal{I}$ such that the supporting strip of the leading line $YX$ touches $K$ in the endpoints $z,w$ of an isometric diagonal. Conversely, for every isometric diameter $[z,w]$ there is a line $YX$ with $X,Y\in\mathcal{I}$ for which  the respective touching points of the supporting strip of $YX$ are $z$ and $w$. Moreover for corresponding pairs $X,Y\in \mathcal{I}$ and $z,w\in\mathrm{bd}(K)$ holds $\mathrm{width}_K(X)\geq d_K(YX)\geq\mathrm{width}_4(z)=\mathrm{width}_4(w)$, with equality in the second inequality if and only if the isometric diameter $[z,w]$ is a double-normal.
\end{statement}

\begin{proof}
Let $Y\in \mathcal{I}$ such that holds the equality $d_K^+(YX)=d_K^-(YX)$. (The continuity property of the distance function implies the existence of such line $YX$.) Denote by $z$ and $w$ a pair of touching points lying in the distinct bounding hypercycles. Then $z$ and $w$ are opposite points in the sense of \cite{jeronimo-castropaper}. In fact, the triangles  $zT_zM_{\triangle}$ and $wT_wM_{\triangle}$ in the left hand figure of Fig. \ref{fig:jeronimo-castro_width} are congruent so the chord $[z,w]$ has equal alternate angle with the corresponding tangent lines $t_z$ and $t_w$. Since these lines disjoint to this strip that their distance is greater than the width of it. Equality holds if and only if the isometric diameter $[z,w]$ is a double-normal of $K$.

\begin{figure}[ht]
\includegraphics[scale=0.6]{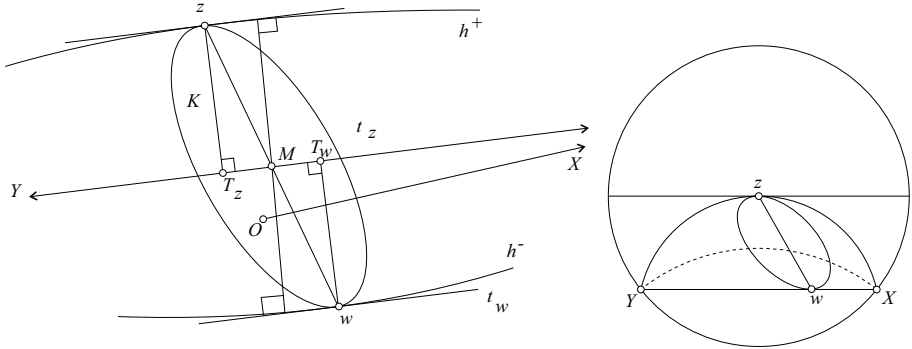}
\caption{Illustration of the proof of Statement\ref{st:isometricdiam}}\label{fig:jeronimo-castro_width}
\end{figure}

Conversely, we prove that every isometric diameter of a h-convex body can be get with the above method. Precisely, if $[z,w]$ is an isometric diameter there is a line $YX$ for which $\mathrm{dist}(z,XY)=\mathrm{dist}(w,XY)$. See the picture on the right hand side of Fig. \ref{fig:jeronimo-castro_width}. This shows the isometric diameter $[z,w]$ in the Poincare's disk model. Let $X$ and $Y$ be the respective intersections of the Euclidean tangent line at the point $w$ and the model circle. Then the Euclidean line $XY$ represents an hypercycle of the hyperbolic plane corresponding to the leading hyperbolic line with ideal points $X$ and $Y$. Using an inversion to the absolute (to the circle of ideal points) we get from this hypercycle its opposite pair which has the same distance from the hyperbolic line $YX$. Inversion sends the Euclidean line $YX$ to a circle which goes through the origin, the intersections with the model circle are $X$ and $Y$ and tangents $K$ at $z$. The only question is that this strip contains or not the body $K$. If $K$ is a h-convex body then its curvature at the point $z$ is greater than the curvature of the hypercycle $YzX$ and the strip contains the body, as we stated.
\end{proof}

Note that if $K$ is a polygon the two endpoints of a double-normal typically cannot be considered to the touching pair of points of a strip which is bounded by hypercycles of the same leading line.

We call a strip \emph{symmetric to its leading line} if the boundary hypercycles have the same distance from the leading line.

\begin{corollary}\label{cor:equivalence}
Let $K$ be a h-convex compact body. Assume that the strip with leading line $YX$ touches $K$ at the points $z$ and $w$, respectively. From the proof of Lemma \ref{st:isometricdiam} we can see easily that the following properties are equivalent to each other:
\begin{itemize}
\item $[z,w]$ is a double normal of $K$
\item $[z,w]$ is orthogonal to $YX$.
\end{itemize}
When this orthogonality property holds the chord $[z,w]$ is an isometric diameter.
If $[z,w]$ isn't orthogonal to $YX$ then it is an isometric diameter if and only if the strip symmetric with respect to the leading line $YX$.
\end{corollary}

\begin{statement}\label{st:widthfunctionofasegment}
Let $AB$ be a segment of length $d$ and take its figure in a projective model on the horizontal diameter in symmetric position (see Fig.\ref{fig:widthofasegment}). The corresponding ideal points of the line $AB$ are $X_{{AB}}$ and $X_{{BA}}$, respectively. Denote by $X_1,X_2\in \mathcal{I}$ those ideal points for which $X_1BA_{\angle}=X_2AB_{\angle}=\pi/2$. We say that the ideal points $X'$ and $X''$ holds the inequality $X'<X''$, if there are on the upper half of the model and in counterclockwise direction the point $X'$ precedes the point $X''$. Finally, denote by $\alpha$ the angle $XAB_{\angle}$ for $X<X_1$ and by $\beta$ the angle $XBA_{\angle}$ if $X>X_2$. Then the width function of the segment $AB$ is
\begin{equation}\label{eq:segmentwidth}
\mathrm{width}_{AB}(X)=\left\{\begin{array}{ccc}
                                      0 & \mbox{ if } & X=X_{AB} \\
                                     \sinh^{-1}\left(\frac{\sin\alpha \sinh d}{\cosh d-\cos\alpha \sinh d}\right)  & \mbox{ if } & X_{AB}<X\leq X_1 \\
                                      d & \mbox{ if } &  X_1<X\leq X_2 \\
                                      \sinh^{-1}\left(\frac{\sin\beta \sinh d}{\cosh d-\cos\beta \sinh d}\right) & \mbox{ if } & X_2<X<X_{BA}\\
                                      0 & \mbox{ if } & X=X_{BA}.
                                    \end{array} \right.
\end{equation}
The width function can be extend symmetrically to the other directions of the plane.
\end{statement}

\begin{proof}
Let $K$ be the closed segment $AB$. Let $X\in \mathcal{I}$ arbitrary ideal point and take those lines $YX$ which belong to the sector bounding by $AX$ and $BX$. The triangle $ABX_\triangle$ has one ideal vertex at $X$ so it can be determined by the length $d=AB$ and the angle $\alpha=BAX_\angle$. Assume that $BAX_\angle \le ABX_\angle$. A line $YX$ should take into consideration if its angle $XMB_\angle$ is greater than $\alpha$ and is less or equal to $XB{X_{AB}}_\angle$ (see Fig. \ref{fig:widthofasegment}).

\begin{figure}[ht]
\includegraphics[scale=0.6]{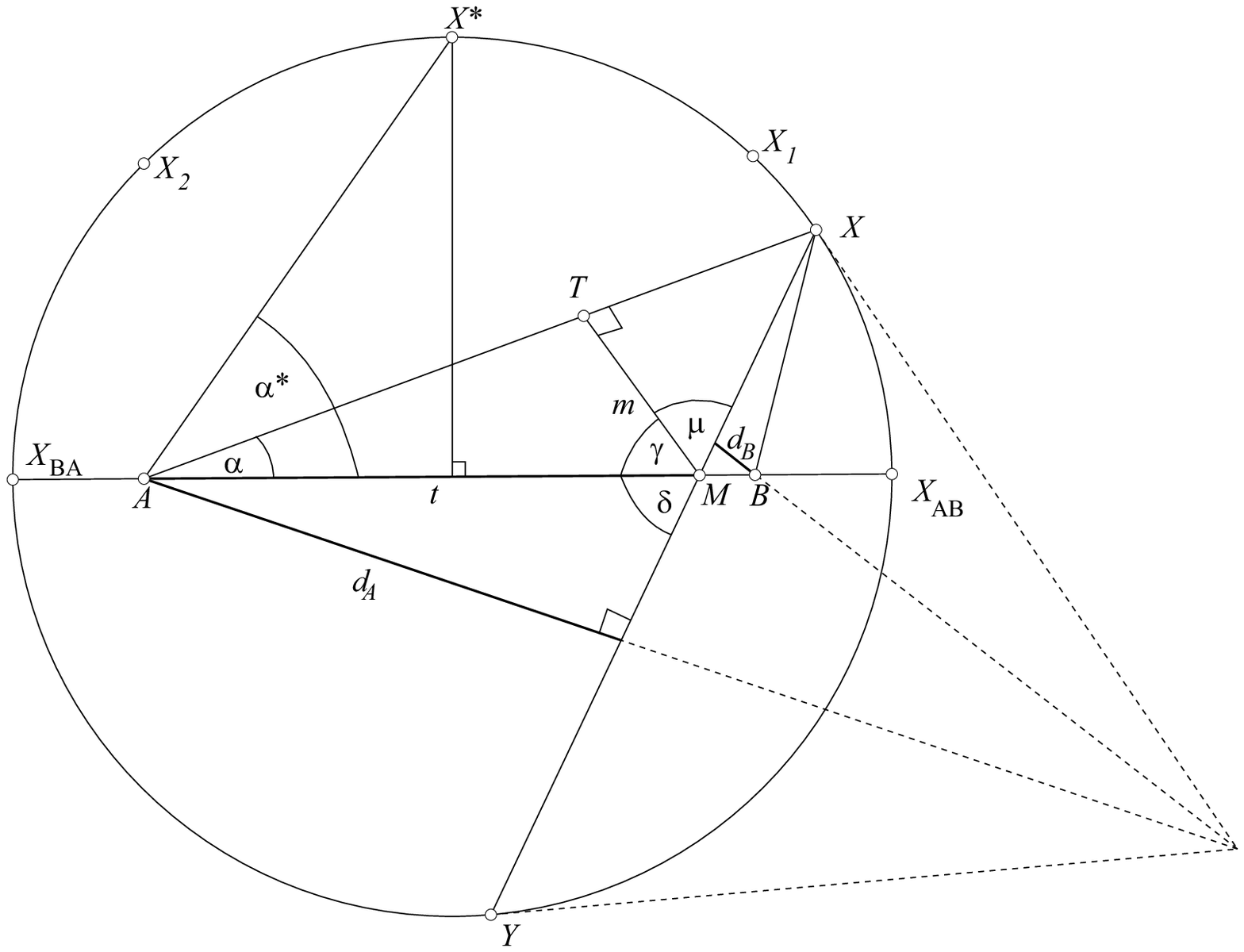}
\caption{}\label{fig:widthofasegment}
\end{figure}

Clearly, for fixed point $X$ we have to determine the sum of the distances $d_A+d_B$ of the points $A$ and $B$ from the line $YX$ and have to take the supremum of these values varying the length $t=|AM|$ between the bounds $0$ and $d$. To this we introduce some notation. Let $X^\star$ be the upper ideal point of the bisector of the segment $AB$. Then the parallel angle $\alpha^\star$ holds the equality
$$
\cos \alpha^\star =\tanh \frac{d}{2}.
$$
We have two symmetries in the calculation, the first one is with respect to the line of $AB$ and the second one is with respect to the bisector of $AB$, meaning that we can assume that $\alpha\leq \alpha^\star$. Denote by $t$ the length of the segment $AM$ and consider a perpendicular from $M$ to $AX$. Its length is $m$. This segment divides the angle $AMX_\angle$ into two parts; one of them say $\gamma$ is the second acute angle of a right triangle, the other one is the parallel angle $\mu$ corresponding to the distance $m$. Hence we have
$$
\cos \mu=\tanh m \quad  \quad \sin \mu=\frac{1}{\cosh m}.
$$
From the right triangle $AMT_\triangle$ we get the equalities
$$
\cos \gamma = \frac{\tanh m}{\tanh t} \quad  \quad \sin \gamma =\frac{\sqrt{\tanh ^2t-\tanh ^2 m}}{\tanh t}.
$$
By the further notation of the figure, we get
$$
\sin \delta =\sin (\pi-\delta)= \frac{\tanh m}{\cosh m\tanh t}+\frac{\tanh m\sqrt{\tanh ^2t-\tanh ^2 m}}{\tanh t}.
$$
Finally, from the right triangle $AMT_\triangle$ we get
$$
\sinh m=\sin \alpha \sinh t \quad \cosh m=\sqrt{1+\sin^2\alpha \sinh^2 t} \quad \tanh m=\frac{\sin \alpha \sinh t}{\sqrt{1+\sin^2\alpha \sinh^2 t}}.
$$
Now we can deduce the following formula
$$
\sin \delta =\frac{\sin \alpha (\cosh t+\cos\alpha \sinh t)}{1+\sin^2\alpha \sinh^2t}=\frac{\sin \alpha} {\cosh t-\cos\alpha \sinh t}.
$$
Our function is
\begin{equation}\label{eq:functionf}
f(t):=d_A+d_B=\sinh^{-1}\left(\sin \delta \sinh t\right)+\sinh^{-1}\left(\sin \delta \sinh (d-t)\right)=
\end{equation}
$$
\sinh^{-1}\frac{\sin \alpha \sinh t}{\cosh t-\cos\alpha \sinh t}+\sinh^{-1}\frac{\sin \alpha \sinh (d-t)}{\cosh t-\cos\alpha \sinh t}.
$$
The derivative of $f(t)$ is
$$
f'(t)=\frac{\sin \alpha }{\left(\cosh t-\cos\alpha \sinh t\right)\sqrt{\left(\cosh t-\cos\alpha \sinh t\right)^2+\sin^2\alpha \sinh^2 t}}+
$$
$$
\frac{\sin\alpha \left(-\cosh d+\cos \alpha \sinh d\right)}{\left(\cosh t-\cos\alpha \sinh t\right)\sqrt{\left(\cosh t-\cos\alpha \sinh t\right)^2+\sin^2\alpha \sinh^2 (d-t)}}.
$$
Since $\cosh t-\cos\alpha \sinh t>0$ holds for all positive $t$ and also we assumed that $\sin\alpha >0$, the condition $f'(t)=0$ is equivalent to the following equality:
$$
\cosh d-\cos \alpha \sinh d=\frac{\sqrt{\left(\cosh t-\cos\alpha \sinh t\right)^2+\sin^2\alpha \sinh^2 (d-t)}}{\sqrt{\left(\cosh t-\cos\alpha \sinh t\right)^2+\sin^2\alpha \sinh^2 t}}.
$$
Introducing the notation $x(t):=\cosh t-\cos\alpha \sinh t$ we can write
$$
x^2(d)-1=\frac{\sin^2\alpha \left(\sinh ^2(d-t)-\sinh^2 t\right)}{x^2(t)+\sin^2\alpha \sinh^2t}.
$$
This leads to
$$
\left(\sinh^2d(1+\cos^2\alpha)-\cos\alpha \sinh (2d)\right)\left(\cosh 2t-\cos\alpha \sinh 2t\right)=
$$
$$
\sin^2\alpha \left(\sinh^2d\cosh(2t)-\sinh(2t)\sinh d\cosh d\right),
$$
from which straightforward calculation shows the following equivalent form:
$$
0=\left(\cosh d-\cos\alpha \sinh d\right)\left((1+\cos^2\alpha)\sinh (2t)-2\cos\alpha \cosh (2t)\right).
$$
Since the first term is always positive we get that the only zero is attend at
$$
\tanh (2t)=\frac{2\cos\alpha}{1+\cos^2\alpha}.
$$
From this $\tanh t=\cos\alpha$, therefore the function $f(t)$ has one maximum at the point $t=\tanh^{-1}\left(\cos\alpha\right)$. It can be seen that at this value $\delta =\pi/2$, so the triangle $AXM_\triangle$ is a right one. Of course, the fix value $d$ sometimes less than that $t$ where the maximum is attend, in this case from that $t$'s which we allowed, $t=d$ produces the larger widthness.
\end{proof}

\section{Bodies of constant width}

One of the most interesting question with respect to the width function is the characterization problem of bodies of constant width. In this section we investigate this question. First we determine the with function of an important body, the with function of a symmetric bounded hypercycle domain.

\begin{statement}\label{st:hypercycledomain}
The width function of the symmetric hypercycle domain $ABCD$ is a combination of the width functions of the "diagonal chords" $AC$ and $BD$. Using the notation of Fig.\ref{fig:hypercycledomain}  we have on the arc $X_0\leq X\leq X'$ the following function
\begin{equation}\label{eq:hypercycledomainwidth}
\mathrm{width}_{K}(X)=\max\{\mathrm{width}_{AC}(X), \mathrm{width}_{BD}(X)\}=\left\{\begin{array}{ccc}
                                      \sinh^{-1}\frac{\tanh h(\cosh a+\sinh a)}{\sqrt{\cosh^2a\cosh^2h-1}} & \mbox{ if } & X_0<X\leq X_1 \\
                                       d & \mbox{ if } &  X_1<X\leq X_2 \\
                                       \tanh^{-1}\cos(\gamma+\delta) & \mbox{ if } & X_2<X<X'
                                     \end{array} \right. .
\end{equation}
The complete function has rotational symmetry of order 4.
\end{statement}

\begin{proof}
Let $K$ be a convex, compact body and let $X$ be a direction. Then we have
$$
\mathrm{width}_K(X)=\sup\{ \mathrm{width}_{AB}(X) \quad : \quad A,B\in K\}.
$$
From this we can deduce the width function of a symmetric hypercycle domain $D$. Let its height $2h$ meaning that its leading line has distance $h$ from the points of the hypercycle arcs, respectively (see the left picture in Fig.\ref{fig:hypercycledomain}).
\begin{figure}[ht]
\includegraphics[scale=0.6]{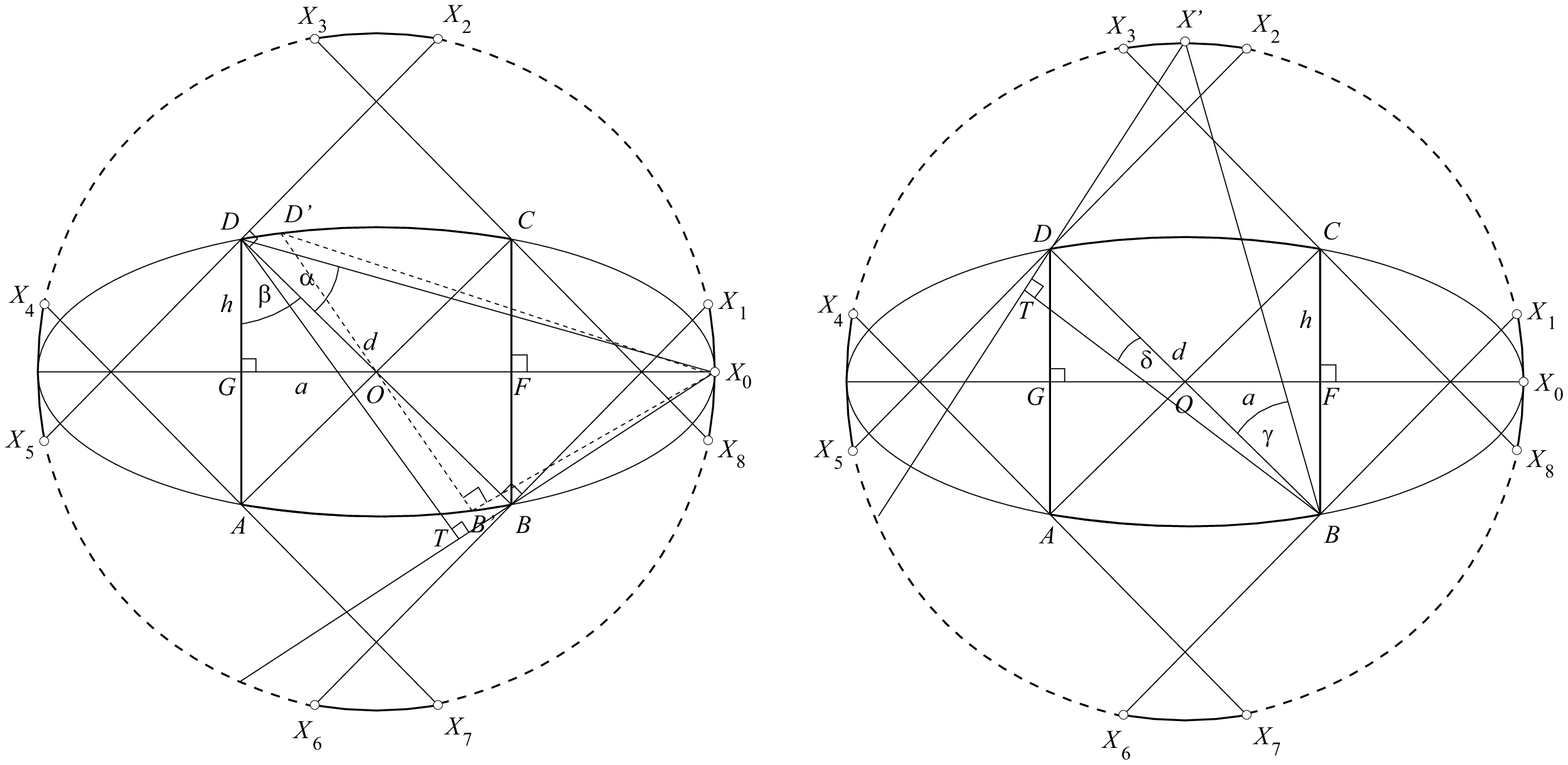}
\caption{}\label{fig:hypercycledomain}
\end{figure}
The chord $FG$ on the leading line has length $2a$. Clearly $AC$ and $BD$ are the longest chords in the domain, they common length is $d=2\cosh^{-1}\left(\cosh a\cosh h\right)$. Observe that any hypercycle domain which contains these two segments also contains those segments which one endpoint is  $AB$ and the other one is in $CD$. Hence we have to concentrate those segments which ends lying in the hypercycle arcs.
From those ideal points which  are on the dashed arcs on Fig.\ref{fig:hypercycledomain} the value of the width is $d$ by Lemma\ref{st:widthfunctionofasegment}. The question is the following: What can we say about the widthness from the points on the bold arcs of the model circle? Clearly, we have $\mathrm{width}_D(X_0)\geq 2h$. On the other hand we can calculate the width of the segment $BD$ from this direction, too. In fact, Lemma\ref{st:widthfunctionofasegment} says, that
$$
\mathrm{width}_{BD}(X)= \sinh^{-1}\left(\frac{\sin\alpha \sinh d}{\cosh d-\cos\alpha \sinh d}\right),
$$
where $\alpha+\beta$ is the parallel-angle of the distance $h$ and $\sin\beta=\frac{\sinh a}{\sinh d/2}$. Since $\cos(\alpha+\beta)= \tanh h$ we get
$$
\tanh h=\cos \alpha\cos \beta -\sin\alpha \sin\beta=\cos \alpha \sqrt{1-\frac{\sinh^2 a}{\sinh^2 \frac{d}{2}}}-\sin\alpha \frac{\sinh a}{\sinh \frac{d}{2}}.
$$
By the hyperbolic theorem of cosine $\cosh\frac{d}{2}=\cosh a\cosh h$ and we get that
$$
\sin\alpha =\frac{\tanh h(\cosh a-\sinh a)}{\sinh\frac{d}{2}},
$$
which gives that
$$
\cos \alpha =\frac{\cosh a\sinh^2h +\sinh a}{\cosh h\sinh \frac{d}{2}}.
$$
We can use formula (\ref{eq:segmentwidth}) to determine the width:
$$
\mathrm{width}_{BD}(X_0)=\sinh^{-1}\frac{\sin\alpha \sinh d}{\cosh d-\cos\alpha \sinh d}.
$$
Long but straightforward calculation shows that
$$
\mathrm{width}_{BD}(X_0)=\sinh^{-1}\frac{\tanh h(\cosh a+\sinh a)}{\sinh\frac{d}{2}}=\sinh^{-1}\frac{\tanh h(\cosh a+\sinh a)}{\sqrt{\cosh^2a\cosh^2h-1}}.
$$
For fixed value $h$ we can use the above formula to determine the width of the chord $B(t)D(t)$ as a function of $a(t)$. The least possible value of $a(t)$ attend at that $a'$ when the corresponding ${B'D'X_0}_\angle$ is the parallel angle of the value $d'$. From this we get
$$
\frac{\tanh h(\cosh a'-\sinh a')}{\sinh\frac{d'}{2}}=\sin\alpha'=\frac{1}{\cosh d'}=\frac{1}{\cosh^2 \frac{d'}{2}+\sinh^2 \frac{d'}{2}}.
$$
If $a(t)\leq a'$ from (\ref{eq:segmentwidth}) we get that
$$
d(t)=\mathrm{width}_{B(t)D(t)}(X_0)\leq \sinh^{-1}\frac{\tanh h(\cosh a'+\sinh a')}{\sinh\frac{d'}{2}}=\sinh^{-1}\left(\frac{1}{\cosh d'}+2\frac{\sinh a'}{\sinh\frac{d'}{2}}\right).
$$
The width a chord from $X_0$ which connects two points of the opposite hypercycle is equal to the width from $X_0$ of such a chord which goes through $O$. So we get
$$
\mathrm{width}_{K}(X_0)= \mathrm{width}_{BD}(X_0)=\sup\left\{\mathrm{width}_{B(t)D(t)}(X_0) : 0\leq a(t)\leq a\right\}.
$$
The same true for the other points of the arc $X_0X_1$.

Assume that $X_2<X<X_3$ holds for the point $X$ (see the right picture in Fig. \ref{fig:hypercycledomain}). We can observe that from these points the width of $K$ is always equal to the width of one of the segments $AC$ and $BD$. On the figure, we drew that situation when
$\mathrm{width}_{BD}(X)=BT$ implying the equality
$$
\tanh (\mathrm{width}_{BD}(X))=\tanh BT =\cos(\gamma+\delta),
$$
where we used the notation of the figure. This proves the statement.
\end{proof}

\begin{figure}[ht]
\includegraphics[scale=0.6]{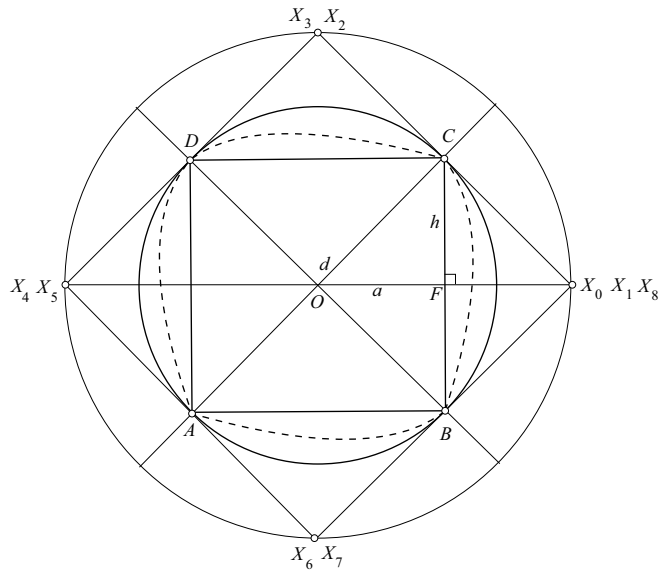}
\caption{Bodies of constant width}\label{fig:constwidthcurves}
\end{figure}

\begin{defi}\label{def:constanwidth}
We say that the compact convex body $K$ has \emph{constant width} if there is a $\lambda\in \mathbb{R}^+$ that for every ideal point $X$ $\mathrm{width}_K(X)=\lambda$ holds.
\end{defi}

\begin{remark}\label{rem:constwidthcurves}
There is an interesting consequence of the Statement \ref{st:hypercycledomain}. When the arcs $\{X : X_8\leq X\leq X_1\}$ and $\{X : X_2\leq X\leq X_3\}$ degenerate to a point the width function of the hypercycle domain has constant value. In this case, $\tanh (d/2)=\cos \pi/4=\sqrt{2}/{2}$, from which
$\cosh^2d/2-1=\sinh^2 d/2=(1/2)\cosh^2d/2$ implying that $\cosh d/2=\sqrt{2}$ and  the unique value of the width function is $d=2\cosh^{-1}(\sqrt{2})=2\ln (\sqrt{2}+1)$. Since $\sqrt{2}=\cosh a\cosh h$ and also $\sinh h=\sqrt{2}/{2}\sinh d/2=\sqrt{2}/{2}$ we get that $a=\cosh^{-1}\sqrt{\frac{4}{3}}$ and $h=\cosh^{-1}\sqrt{\frac{3}{2}}$. This phenomenon exclude the characterization of bodies of constant width, because each convex body belonging to a circular disk of diameter $d=2\ln (\sqrt{2}+1)$ and containing the endpoints of a pair of orthogonal diameters is a convex body of constant width. (E.g. see in Fig. \ref{fig:constwidthcurves} the circle, the quadrangle with equal angles and sides, and the body which bounded by the dotted curve, respectively.)
\end{remark}

The following statement generalises this observation.

\begin{statement}\label{st:regularpoly}
Let $k\geq 2$ an integer number. Assume that the radius $r$ of the circumscribed circle of a regular $n$-polygon $P$ holds
$$
\tanh r \geq  \tanh r_{\min}:=\left\{\begin{array}{ccc}
      \frac{-1+ \sqrt{1+4\sin\frac{\pi}{2k-1}\cot\frac{(k-1)\pi}{2k-1}}} {2\cos\frac{(k-1)\pi}{2k-1}}  & \mbox{ if } & n=2k-1,\\
             & & \\
      \frac{1}{\sqrt{\cos ^2\frac{(k-1)\pi}{2k} +\tan^2\frac{(k-1)\pi}{2k}}} & \mbox{ if } & n=2k.
\end{array}\right.
$$
Then $P$ is a body of constant width $d$, where
$$
\sinh \frac{d}{2}=
\left\{\begin{array}{ccc}
                      \frac{\sinh r}{\sin\frac{(k-1)\pi}{2k-1}} & \mbox{ if } & n=2k-1,\\
                      & & \\
                      \frac{\sinh r}{\sin \frac{(k-1)\pi}{2k}}  & \mbox{ if } & n=2k.
\end{array}\right.
$$
\end{statement}

\begin{proof}
If $n=2k$ is an even number then we can use the train of though of Remark \ref{rem:constwidthcurves}.
\begin{figure}[ht]
\includegraphics[scale=0.6]{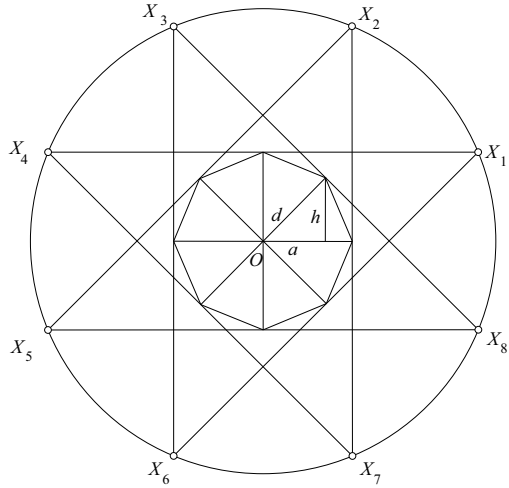}
\caption{Regular $2k$-polygon of constant width}\label{fig:constwidthregcurves}
\end{figure}
Put the polygon into the center of a projective model with a large radius similarly as you can see on Fig. \ref{fig:constwidthregcurves}. Consider those pairs of support lines of $P$ which go through the opposite vertices of a diameter. This $2k$ bands intersect the boundary of the model those ideal points from which at least one diameter has width $d$. We determine the radius of the model such a way that the getting arcs cover the circle but not overlap. The common points are $X_i$ $i=1,\cdots 2k$. The endpoints of one support line are $X_iX_{i+n-1}$ where we take the indices modulo $2k$. The central angle of such a line is $\frac{(k-1)}{k}\pi$ thus the parallel angle corresponds to $d/2$ is $\frac{(k-1)}{2k}\pi$. Hence
$$
\tanh \frac{d}{2}=\cos \frac{(k-1)\pi}{2k}
$$
from which we get that
$$
\tanh d=\frac{2\cos \frac{(k-1)\pi}{2k}}{1+\cos ^2\frac{(k-1)\pi}{2k}}
$$

The corresponding radius $r$ of the circumscribed circle holds:
$$
\sinh r=\frac{\sinh \frac{d}{2}}{\sin \frac{(k-1)\pi}{2k}}=\frac{\frac{\tanh \frac{d}{2}}{\sqrt{1-\tanh ^2\frac{d}{2}}}}{\sin \frac{(k-1)\pi}{2k}}=
\frac{\cos \frac{(k-1)\pi}{2k}}{\sin ^2\frac{(k-1)\pi}{2k}},
$$
or
$$
\tanh r=\frac{\cos \frac{(k-1)\pi}{2k}}{\sqrt{\sin ^4\frac{(k-1)\pi}{2k}+\cos^2\frac{(k-1)\pi}{2k}}}=\frac{\cos \frac{(k-1)\pi}{2k}}{\sqrt{1-\sin ^2\frac{(k-1)\pi}{2k}\cos ^2\frac{(k-1)\pi}{2k}}}=\frac{1}{\sqrt{\cos ^2\frac{(k-1)\pi}{2k} +\tan^2\frac{(k-1)\pi}{2k}}}.
$$
The odd case is more complicated. A diametral chord doesn't contain the origin. It determines a set of ideal points from which the widthness of the chord is $d$. In the model these sets are the union of two non-congruent arcs, the longest one is in that half-plane of the chord which contains the origin. We guarantee the constant widthness property if the rotated copies ($n$-times by the angle $2\pi/n$) of these sets cover the whole boundary of the model. If the vertices of the polygon are $P_1,\ldots ,P_{2k-1}$, then the diametral chords are the segments $P_iP_{i+k-1}$ respectively. The diametral chords are rotated copies of each other by the integer multipliers of the angle $2\pi/(2k-1)$ around the origin. (See Fig. \ref{fig:oddcase}.)
\begin{figure}[ht]
\includegraphics[scale=0.6]{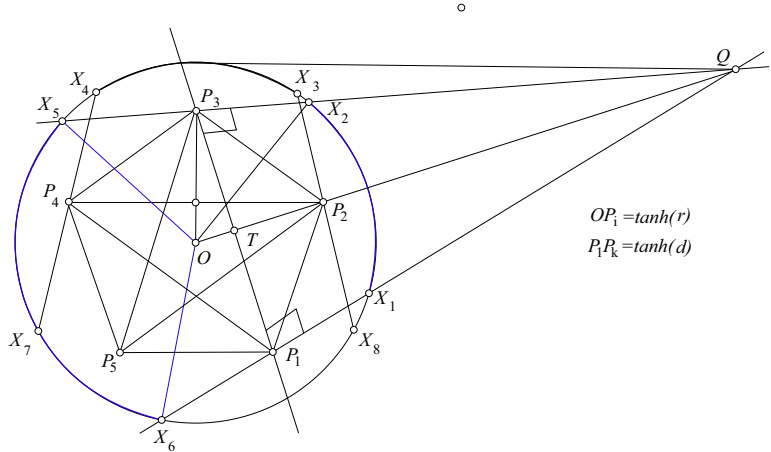}
\caption{Regular $2k-1$-polygon of constant width}\label{fig:oddcase}
\end{figure}
Denote by $X_1X_2\cup X_{2k-1}X_{2k}$ and $X_3X_4\cup X_{2k+1}X_{2k+2}$ the sets of ideal points from which the diametral chords $P_1P_k$ and $P_2P_{k+1}$ can be seen in their common real length $d$. When $X_{2k-1}X_{2k}$ and $X_{2k+1}X_{2k+2}$ overlap the rotational symmetry guarantees that the polygon is a body of constant width $d$, since for every ideal points one of the diametral chord can be seen in its real length $d$. Thus we have to determine the central angle $X_{2k-1}{OX_{2k}}_{\angle}$. The minimal value of $r$ occurs when $X_{2k-1}{OX_{2k}}_{\angle}=2\pi/(2k-1)$.

Since $P_3OQ_{\angle}= \frac{(k-1)\pi}{2k-1}$ and the radius of the model is $1$, the Euclidean distance of the diametral chords to the origin is $OT=\cos \frac{(k-1)\pi}{2k-1}\tanh r$. The pole of the chord containing the segment $P_1P_k$ is $Q$ so $OQ=(\cos \frac{(k-1)\pi}{2k-1}\tanh r)^{-1}$.
Let denote by $\alpha= \frac{(k-1)\pi}{2k-1}$ and $\beta:=OQ{P_{k}}_{\angle} $, respectively. Since $Q$, $X_2$ and $X_{2k-1}$ are collinear points we have
$QX_2\cdot QX_{2k-1}=OQ^2$ where $QX_2$ and $QX_{2k-1}$ are the roots of the polynomial equation of second order:
$$
0=x^2-2OQ\cdot x\cos \beta+OQ^2-1.
$$
Hence the two roots are $x_{1,2}=OQ\cos\beta\pm \sqrt{1-\sin^2\beta OQ^2}$, so $QX_{2k-1}=OQ\cos\beta+\sqrt{1-\sin^2\beta OQ^2}$. Now if the half of the central angle of the arc $X_{2k-1}X_{2k}$ is equal to the half of the central angle $P_1{OP_2}_{\angle}=\frac{2\pi}{2k-1}=\frac{\alpha}{k-1}$, then we get the equality:
$$
\sin \frac{\alpha}{k-1}=\sin (\pi-\frac{\alpha}{k-1})=QX_{2k-1} \sin\beta=\sin\beta \left(OQ\cos\beta+\sqrt{1-\sin^2\beta OQ^2}\right).
$$
From this equality we can determine again $OQ$ and get
$$
0=OQ^2-2\sin\frac{\alpha}{k-1}\cot \beta OQ+\frac{\sin^2\frac{\alpha}{k-1}}{\sin^2\beta}-1,
$$
from which
$$
OQ=\sin\frac{\alpha}{k-1}\cot \beta \pm \cos \frac{\alpha}{k-1}.
$$
From the Euclidean right triangle ${TQP_k}_\triangle$ we have that
$$
\cot \beta=\frac{OQ-OT}{\sin\alpha \tanh r}=\frac{\frac{1}{\cos\alpha \tanh r}-\cos\alpha \tanh r}{\sin\alpha \tanh r}= \frac{1-\cos^2\alpha \tanh^2 r}{\cos\alpha \sin\alpha \tanh^2 r},
$$
and get the equality
$$
\frac{1}{\tanh r\cos\alpha}=\sin\frac{\alpha}{k-1}\frac{1-\cos^2\alpha \tanh^2 r}{\cos\alpha \sin\alpha \tanh^2 r}\pm \cos \frac{\alpha}{k-1}.
$$
From this equation we get either the equality $(\star)$:
$$
(\star) \qquad 0=\tanh^2r \sin\frac{(k-2)\alpha}{k-1}\cos\alpha-\tanh r \sin\alpha +\sin\frac{\alpha}{k-1}
$$
or the equality $(\star \star)$:
$$
(\star \star) \qquad 0=\tanh^2r \sin\frac{k\alpha}{k-1}\cos\alpha+\tanh r \sin\alpha +\sin\frac{\alpha}{k-1}.
$$
The corresponding solutions are
$$
(\star) \qquad \tanh r= \frac{\sin\alpha\pm \sqrt{\sin^2\alpha-4\sin\frac{(k-2)\alpha}{k-1}\sin\frac{\alpha}{k-1}\cos\alpha}} {2\sin\frac{(k-2)\alpha}{k-1}\cos\alpha}
$$
and
$$
(\star \star) \qquad \tanh r= \frac{-\sin\alpha\pm \sqrt{\sin^2\alpha+4\sin\frac{k\alpha}{k-1}\sin\frac{\alpha}{k-1}\cos\alpha}} {2\sin\frac{k\alpha}{k-1}\cos\alpha},
$$
respectively. Substituting $\alpha=\frac{(k-1)\pi}{2k-1}$ we get
$$
(\star) \qquad \tanh r=\frac{\sin\frac{(k-1)\pi}{2k-1}\pm \sqrt{\sin^2\frac{(k-1)\pi}{2k-1}-4\sin\frac{(k-2)\pi}{2k-1}\sin\frac{\pi}{2k-1}\cos\frac{(k-1)\pi}{2k-1}}} {2\sin\frac{(k-2)\pi}{2k-1}\cos\frac{(k-1)\pi}{2k-1}}
$$
and
$$
(\star \star) \qquad \tanh r=\frac{-\sin\frac{(k-1)\pi}{2k-1}\pm \sqrt{\sin^2\frac{(k-1)\pi}{2k-1}+4\sin\frac{k\pi}{2k-1}\sin\frac{\pi}{2k-1}\cos\frac{(k-1)\pi}{2k-1}}} {2\sin\frac{k\pi}{2k-1}\cos\frac{(k-1)\pi}{2k-1}}.
$$
Among these solutions there is only one which have to take into consideration. In $(\star)$ (by the inequality $\tanh r<1$) we need that
$$
\left(\sin\alpha-2\sin\frac{(k-2)\alpha}{k-1}\cos\alpha\right)<\mp\sqrt{ \sin^2\alpha-4\sin\frac{(k-2)\alpha}{k-1}\sin\frac{\alpha}{k-1}\cos\alpha},
$$
from which the inequality:
$$
-4\sin\alpha\sin\frac{(k-2)\alpha}{k-1}\sin\frac{\alpha}{k-1}\cos\alpha+4\sin^2\frac{(k-2)\alpha}{k-1}\cos^2\alpha< -4\sin\frac{(k-2)\alpha}{k-1}\sin\frac{\alpha}{k-1}\cos\alpha
$$
holds. Simplifying it we get
$$
\sin\frac{(k-2)\alpha}{k-1}\cos\alpha<\sin\alpha\sin\frac{\alpha}{k-1}-\sin\frac{\alpha}{k-1}=\left(\sin\alpha-1\right)\sin\frac{\alpha}{k-1},
$$
which gives a contradiction, because the right hand side is negative and the left hand side is positive. Similarly $\tanh r>0$ implies that in $(\star \star)$ we have to take into consideration only the sign plus. Hence we have that
$$
\tanh r=\frac{-\sin\frac{(k-1)\pi}{2k-1}+ \sqrt{\sin^2\frac{(k-1)\pi}{2k-1}+4\sin\frac{k\pi}{2k-1}\sin\frac{\pi}{2k-1}\cos\frac{(k-1)\pi}{2k-1}}} {2\sin\frac{k\pi}{2k-1}\cos\frac{(k-1)\pi}{2k-1}}=\frac{-1+ \sqrt{1+4\sin\frac{\pi}{2k-1}\cot\frac{(k-1)\pi}{2k-1}}} {2\cos\frac{(k-1)\pi}{2k-1}},
$$
proving the last statement.
\end{proof}

\begin{remark}\label{rm:constwidthquadrangle}
For $k=2$ we have
$$
\tanh r_{\min}:=
\left\{\begin{array}{ccc}
            \frac{-1+ \sqrt{1+4\sin\frac{\pi}{3}\cot\frac{\pi}{3}}} {2\cos\frac{\pi}{3}}=\sqrt{3}-1  & \mbox{ if } & n=3,\\
             & & \\
              \frac{1}{\sqrt{\cos ^2\frac{\pi}{4} +\tan^2\frac{\pi}{4}}}=\sqrt{\frac{2}{3}} & \mbox{ if } & n=4,
\end{array}\right.
$$
and the corresponding values of $d_{\min}$ are
$$
\sinh d_{\min}=
\left\{\begin{array}{ccc}
           \frac{2\sqrt{3}-2}{\sqrt{6\sqrt{3}-9}}=1.2408  & \mbox{ if } & n=3,\\
             & & \\
             2 & \mbox{ if } & n=4.
\end{array}\right.
$$
Observe that the respective functions $\tanh r_{\min}(2k-1)$ and $\tanh r_{\min}(2k)$ of $k$ are strictly decreasing tends to zero in the variable $k$. The question is that there is or not a value of $k$ when the difference of the two series change their sign or for every pairs of successive integers $2k-1$ and $2k$ this difference is negative. This question can be answered by computer easily. The first six values of the two functions can be found in the table below showing that the merged series is monotone decreasing up to the index $9$.
\begin{table}[ht]
  \centering
  \begin{tabular}{|c||c|c|c|c|c|c||c|c|c|c|}
     \hline
           n & 3 & 4 & 5 & 6 & 7 & 8 & 9 & 10 & 11 & 12 \\ \hline
     $\tanh r$ & 0.7321 & 0.8165 & 0.5309 & 0.5547 & 0.4080 & 0.4091 & 0.3286 & 0.3233 & 0.2739 & 0.2673 \\
     \hline
   \end{tabular}\vspace{3mm}
     \caption{The sequence of the radiuses.}\label{table:radiuses}
\end{table}
\end{remark}

An interesting phenomenon of constant width property that it depends also on the size of a body, not only on the combinatorial and symmetry properties of it. In spite of that there are bodies of constant width which size is non-essential in this point of view.

\begin{statement}\label{st:circle}
The hyperbolic circle and the hyperbolic Reuleaux polygons are convex bodies of constant width independently on their sides.
\end{statement}

\begin{proof}
Consider the projective or Cayley-Klein disk model of the hyperbolic plane and consider a hyperbolic circle which center is the centre of the model disk. This is an Euclidean circle in the model. Every leading line $YX$ is a chord of the model and the corresponding hypercycles are such half-ellipses  which are tangent to the model circle at the points $Y$ and $X$, respectively; and touching the given circle (see the left hand side of Fig. \ref{fig:circle}).

\begin{figure}[ht]
\includegraphics[scale=0.6]{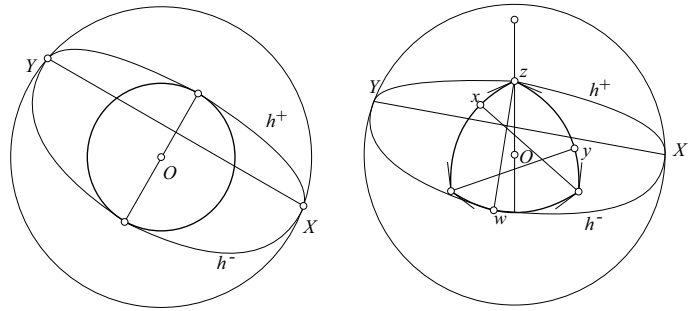}
\caption{Hyperbolic circles and Realeaux polygons are bodies of constant width}\label{fig:circle}
\end{figure}

From the rotational symmetry of the figure we can see that for all line $YX$ we get the same width $d_K(YX)$ which is equal to the hyperbolic diameter of the circle. From this the statement is obvious for the circle.

Consider the case of the Realeaux polygons. From the definition (see e.g. \cite{santalo} p.410) follows that the maximal distance $d$ of its points reach at such a pair of points that one of them is a vertex of the polygon and the other lies on the opposite arc of the boundary. Let now $YX$ a leading line and let be the corresponding hypercyles $h^+$ and $h^-$ with the touching points $z$ and $w$, respectively (see the right figure in Fig. \ref{fig:circle}). If $z$ is a vertex and $w$ lies on the opposite arc then there is a normal line at $w$ which also perpendicular to $YX$ and go through $z$. This means that $d_K(YX)=d$. On the other hand if $x$ and $y$ are two smooth points of contact of a strip corresponding to a leading line $ZW$ then the distinct normals at $x$ and $y$ are orthogonal to the line $ZW$. But in a Realeaux polygon every normal goes through the opposite vertex hence every two normals are intersect each other. This is a contradiction because there is no common perpendicular of intersecting lines. Hence $\mathrm{width}_K(X)=d$ for all $X\in \mathcal{I}$ implying that Realeaux-polygons are bodies of constant width.
\end{proof}

\section{Diametral chords and the bodies of constant diameter.}

On the usual way we define the \emph{diameter} of a convex body $K$ as the greater distance between its two points. 

\begin{statement}\label{st:diameter}
The diameter of the compact, convex body $K$ is
$$
\mathrm{diam}(K)=\sup\{\mathrm{width}_K(X): \quad X\in\mathcal{I}\}.
$$
The supremum is a maximum, so there is at least one point $X^\star\in \mathcal{I}$ such that $\mathrm{diam}(K)= \mathrm{width}_K(X^\star)$.
\end{statement}

\begin{proof} Let $z$, $w$ be such points of the boundary of $K$ for which $\mathrm{diam}(K)=d(z,w)$. Clearly, those lines $t^+$ and $t^-$ which are orthogonal to the chord $[z,w]$ at $z$ and $w$ are supporting lines of $K$. If $YX^\star$ a line which orthogonally intersects $[z,w]$ and $h^+$ and $h^-$ are the distance curves corresponding to $YX^\star$ through $z$ and $w$, respectively, then $d_K(YX^\star)=d(z,w)$. Hence $\mathrm{width}_K(X^\star)\geq d_K(YX^\star)=d(z,w)=\mathrm{diam}(K)$. Assume now that $\mathrm{width}_K(X^\star)>\mathrm{diam}(K)$. Then there is an $Y^\star\in \mathcal{I}$ such that $\mathrm{diam}(K)<d_K(Y^\star X^\star)$. Let $\{z^\star , w^\star\}$ be a touching pair of points in the intersection of the boundary of $K$ and that strip which distance is $d_K(Y^\star X^\star)$. Obviously, $d(z^\star,w^\star)\geq d_K(Y^\star X^\star)$ yielding a contradiction. Hence in fact, $\mathrm{diam}(K)=\sup\{\mathrm{width}_K(X): \quad X\in\mathcal{I}\}$, as we stated.
\end{proof}

The above proof shows that the endpoints of a diametral chord are such points of $\mathrm{bd} K$ in which the normals are double-normals. This fact is trivial in all such spaces where orthogonality is defined by metric property. There are three papers of P.V.Ara\'ujo on \emph{metric width} of the bodies of the hyperbolic plane (see  \cite{araujo_1},\cite{araujo_2}, \cite{araujo_3}) in which the concept of curves of constant width was seriously investigated. The author said that a \emph{plane-convex body is of constant with $\delta$} if its diameter is equal to $\delta $ and each point of the boundary is an endpoint  of such a chord which length is equal to $\delta$. We call such a chord to \emph{diametral chords} in this paper. In another terminology the \emph{constant diameter} property was rather used. We use also this latter phrase. The most important properties of constant diameter property (was proved first by Ara\'ujo) we now collect in a statement:

\begin{statement}[\cite{araujo_1},\cite{araujo_2}]\label{sta:constdiamprop}
A body of constant diameter of the hyperbolic plane holds the properties:
\begin{itemize}
\item Every diametral chord is a double normal of the boundary of the body, cutting the boundary curve orthogonally at both ends (more
precisely, the perpendicular line at each extremity of each diametral chord is a line of support of the body).
\item Every chord of that is orthogonal to the boundary at one end is a diametral chord, and therefore it is a double normal.
\item Any two distinct diametral chords must intersect each other.
\end{itemize}
\end{statement}

Ara\'ujo also introduced the concepts of \emph{track function} and \emph{intersection function}, respectively. He fixed a diametral chord and denoted by $\theta$ the oriented angle of another diametral chord with the fixed one (see the oriented angle $CMB$ in Fig. \ref{fig:constdiam}).

\begin{figure}[ht]
\includegraphics[scale=0.6]{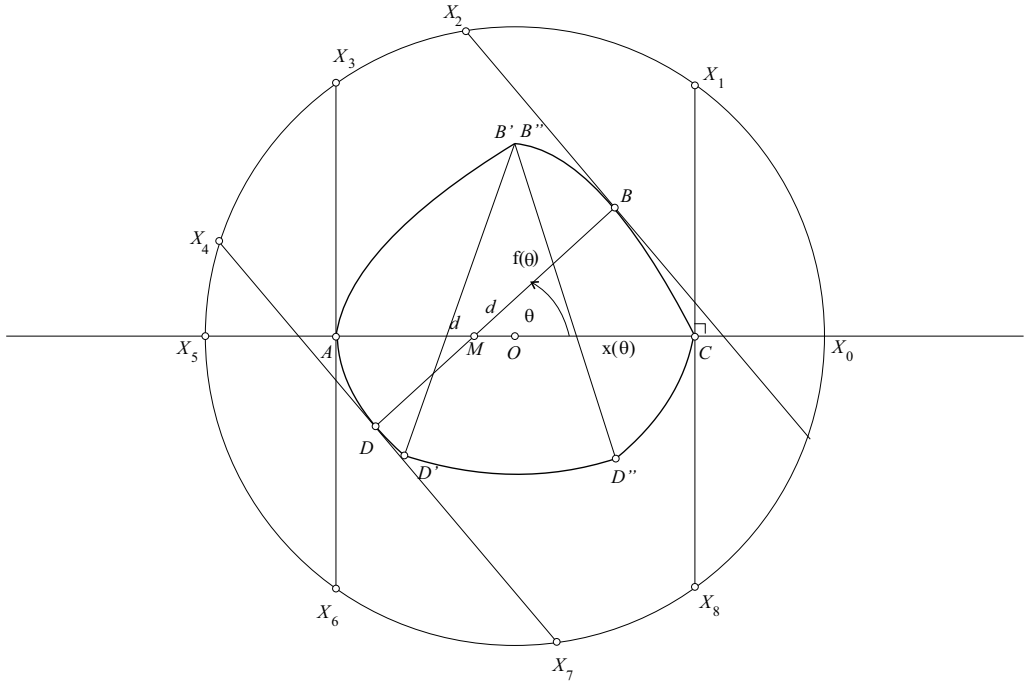}
\caption{The properties of bodies with constant diameter}\label{fig:constdiam}
\end{figure}

When a point from the right end of the fixed chord goes along the boundary then the corresponding diametral chord change its position and angle, respectively. Using the notation of the figure we can see that the boundary curve can be determined by the distance of the point of intersection $M$ to the right end $C$ (which is the value $x(\theta)$ of the track function) and the distance $MB$ which is the value of the intersection function $f(\theta)$. It was proven in \cite{araujo_3} that the original chord $AC$ can be chosen such a way that the following properties hold:
\begin{itemize}
\item Both the track function and the intersection function are continuous on each interval $(0; \pi)$ and $(\pi; 2\pi)$ and for $\theta\in(0,\pi)$ hold that $x(\theta +\pi)=x(\theta)$, $0\leq f(\theta)\leq d$ and $f(\theta)+f(\theta+\pi)=d$.
\item The diametral chord $AC$ can be chosen in such away that the continuity property of these functions extended to the interval $(0,2\pi)$ giving that they are can be extended periodically such that the above properties remain valid on $\mathbb{R}$.
\item For $C^3$ boundary between the track and intersection function there is the connection $f'(\theta)=-x'(\theta)\cos \theta$ implying that the intersection function determines the track function up to the choice of a constant.
\end{itemize}

There is a nice recent result of M. Lassak in \cite{lassak} which connects the spherical concepts of constant width and constant diameter. He proved that a convex body on the two-dimensional sphere of diameter $\delta $ is less than $\pi/2$ is of constant diameter $\delta$ if and only if it is of constant width $\delta$. As we saw in Remark \ref{rem:constwidthcurves} this is not true in the hyperbolic plane, the quadrangle of this remark is a body of constant width but it has only two diametral chords  $AC$ and $BD$, respectively. We now prove the reversal of this statement.

\begin{theorem}\label{thm:constdiam}
If the body $K$ is a convex body of constant diameter $d$ then it is a body of constant width $d$.
\end{theorem}

\begin{proof}
Since the diameter is the maximum value of the width function of $K$ we have to prove that there is no such direction from which the width of $K$ is less than the diameter $d$. Let $AC$ be a diametral chord with endpoints $A$, $C$ and length $d$. Let $YX$ be any line which intersects $AC$ orthogonally. Since $AC$ is a double normal $d\geq \mathrm{width}_K(X)\geq \mathrm{width}_{AC}(X)=d$. Let $\mathcal{X}(AC)$ be the set of ideal points from which the width of $AC$ is equal to $d$. We should prove that $\cup\{\mathcal{X}(AC) \, : \, AC \mbox{ is a diametral chord }\}= \mathcal{I}$. Observe that the tangent lines at the smooth points of the boundary continuously change their position such that the function $\theta(s)$ parameterized by the arc-lengths is strictly increasing and at a corner points each support line determines a diametral chord which is orthogonal to it and one end is the given corner point. From this latter follows that for each $\theta$ there is an unique support line of the body such that the intersection points of the support line with a great circle containing the body produce a continuous bijection between them. This means that the union of the ideal points of the support lines already cover the whole $\mathcal{X}$ implying that it is even more true for $\cup\{\mathcal{I}(AC) \, : \, AC \mbox{ is a diametral chord }\}$.
\end{proof}

We saw that without further condition a body of constant width is not a body of constant diameter. To give a positive result in this a  direction we introduce a new concept.

\begin{defi}
We say that a convex, compact body has the \emph{constant shadow property} if there is a positive number $d$ such that $d_K(YX)=d$ for all line $YX$ which intersects the body $K$.
\end{defi}

Of course a body which has the constant shadow property with the number $d$ then it is a body of constant width with the same $d$. A circle is a body with this property, hence this class is not empty. We prove the following:
\begin{theorem}\label{thm:constdiamandconstshadow}
The convex, compact body $K$ of the hyperbolic plane has the constant shadow property with the number $d$ then it is a body of constant diameter $d$.
\end{theorem}
\begin{proof}
We prove that a body of constant shadow property is a body of constant diameter. Let $d=d_K(YX)$, where $YX\cap K\ne \emptyset$. Then the diameter of $K$ is less or equal to $d$. Really, if $\mathrm{diam }K=d^\star>d$ then there is a chord $PQ$ for which $|PQ|=d^\star$. Let $YX$ is a line which orthogonally intersects the segment $PQ$. Obviously $d_K(YX)\geq d^\star >d$ which is a contradiction.

Consider a point $P$ on the boundary of $K$. Let $t$ be the supporting line of $K$ at $P$; $PQ$ be a chord of $K$ which orthogonal to $t$ and $YX$ be an intersecting line of $K$ perpendicular to $PQ$. Since $d_K(YX)=d$ we have two possibilities. If $PQ$ is a double-normal (at $Q$ it is orthogonal to a supporting line of $K$ through $Q$) then its length is $d$ and it is a chord from $P$ with length $d$. If $PQ$ isn't a double normal $d=d_K(YX)<|PQ^\star|$, where $Q^\star$ is a common point of the another branch of the hypercycle domain of minimal width containing $K$ with fundamental line $YX$. This is impossible because $PQ^\star$ is longest chord as the diameters of $K$.
\end{proof}

\begin{corollary}
For a convex compact body of the plane equivalent the following two properties
\begin{itemize}
\item It has the constant shadow property with number $d$.
\item All diameter of $K$ is a double-normal of $K$ in the general meaning. (We extend the concept of normal to a non-smooth point of the boundary such that it is a line which goes through the point and orthogonal to a supporting line of $K$ at this point.)
\end{itemize}
Theorem \ref{thm:constdiamandconstshadow} and the above equivalence implies the following theorem:

\begin{theorem}\label{theo:doublenormals}
If $K$ is a plane convex compact body of constant shadow property then every normal of $K$ is a double-normal.
\end{theorem}

Since by Theorem 3 in \cite{jeronimo-castropaper} the property that every normal of $K$ is a double-normal is equivalent to the property that $K$ is a body of constant width with respect to the function $\mathrm{widht}_4(\cdot)$, we can deduce that the property of constant shadow knows this property of constant width of Jeronimo-Castro paper.

Further analogises of this theorem well-known in Euclidean geometry (see \cite{soltan} and the references therein). in the hyperbolic plane this statement has been proved for the function $\mathrm{widht}_3(\cdot)$ in \cite{leichtweiss} (Theorem 3.7), and without proof was used for $\mathrm{widht}_1(\cdot)$ in \cite{santalo} (Section 7).
\end{corollary}

\section{The thickness (or width) of a hyperbolic domain}

Usual the width of a convex body $K$ is the minimal width of it with respect to the directions of the plane. In the hyperbolic plane this property can be understood in two ways:
\begin{itemize}
\item In the first possibility we collect the possible values of the width function and from these values we create a common one for example taking their infimum (this leads to the usual definition of width of $K$;
\item in the second possibility we consider the intersecting lines to an independent "directions" (see the concept of constant shadow property), and from these values create the infimum.
\end{itemize}
These two methods lead to the same result, if in the first one we change the concept of width with respect to a direction to something-like thickness of the body changing the supremum to infimum in this definition.

\begin{defi}\label{def:thickness}
Let the \emph{thickness} of the body $K$ in the direction $X$ is
$$
\mathrm{thi}_K(X):=\inf\{d_K(YX)\, : \quad  YX\cap K\ne \emptyset\}
$$
and let \emph{the thickness of $K$} ($\mathrm{thi}(K)$) be the infimum of the values $\mathrm{thi}_K(X)$ take into consideration all direction $X\in\mathcal{I}$.
\end{defi}

By compactness argument we can see easily that there is at least one $X^\star\in\mathcal{I}$, such that $\mathrm{thi}_K(X^\star)=\mathrm{thi}(K)$.

Note that the concept of the thickness for a convex hyperbolic body differs from the concept of the usual width. The quadrangle $ABCD$ in Fig. \ref{fig:constwidthcurves} is a body of constant width with $d=2\ln(\sqrt{2}+1)$ but its thickness is less or equal to $2h=2\ln\left(\sqrt{\frac{3}{2}}+\sqrt{\frac{1}{2}}\right)< d$ (see Remark \ref{rem:constwidthcurves}). Unfortunately, it is also possible that a chord which represents the thickness of the body isn't a double normal. Consider Fig. \ref{fig:thickness}., where we can see a plane-convex body which is the union of two regular triangles (forming a quadrangle with equal sides), two quoters of circle and two quoters of an ellipse. We denoted by dotted lines the bounding hypercycles of the hypercycle domain. It can be seen easily that the thickness of this body is represented by a strip containing the points $u,z$, and $v,w$, respectively.

\begin{figure}[ht]
\includegraphics[scale=0.6]{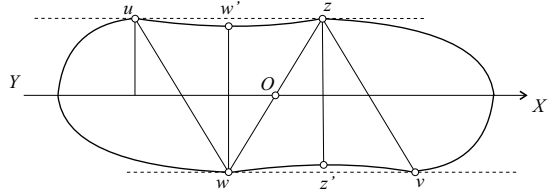}
\caption{A chord which represents the smaller strip distance isn't a double-normal.}\label{fig:thickness}
\end{figure}

Now the respective chords joining the corresponding touching points are $[u,v]$, $[u,w]$, $[z,v]$ and $[z,w]$ but non of them could be double-normal. On the other hand if the ellipse and the hypercycle at $w$ have a common tangent then the chords $[w,w]'$ and $[z,z]'$ are double-normals, respectively. This shows that there are chords which are double-normals and also are isometric diameters (their endpoints are opposite points in the meaning of \cite{jeronimo-castropaper}) but their lengths smaller than the thickness of $K$. This phenomenon is differ from the standard idea of the geometry of a convex body. From Statement \ref{st:isometricdiam} we can see that the reason is the definition of convexity and for h-convex bodies this phenomenon isn't perceptible. An interesting question that a shorter isometric diameter of a h-convex body is or isn't a double normal, and what is the connection with the thickness of the body. Presently we don't know the answer.

To verify that the definition was chosen appropriately we prove a version of Blaschke's theorem in the hyperbolic plane. This theorem says that  every convex compact body in the Euclidean plane of thickness $1$ contains a circle of radius $1/3$ (see example 17 in \cite{jaglomboltyansky} or p.287 in \cite{reiman}). For h-convex bodies we can prove an analogous of this result only . Since in hyperbolic plane there is no similarity, our result depends not only the thickness but also the diameter of the body.

\begin{theorem}\label{thm:blaschke}
Let $K$ be a h-convex body of the hyperbolic plane with diameter $d$ and thickness $t$. Then it contains a circle of radius $r$ for which
\begin{equation}\label{eq:blaschke}
\tanh r=\frac{\sinh t\cosh t}{3\cosh^2d+\sinh^2 t}.
\end{equation}
\end{theorem}

We note that if we use  a first order approximation of the hyperbolic functions this formula leads to the Euclidean one, namely we get  $r=t/3$.

\begin{proof}
Consider a boundary point $P$ of $K$ and a point $Q\in K$. The half-line $\overrightarrow{PQ}$ intersects the boundary of $K$ at the point $Q^\star$. Let $K_P$ be the set containing those points $Q$ of $K$ for which $\tanh |PQ|\leq  \left(1-\frac{1}{3\cosh^2d}\right)\tanh |PQ^\star|$. We prove that $K_P$ is a convex set. Observe that a shrinking map which any point $X$ of the  plane sends to that point $X'$ of the half-line $PX$ for which $\tanh |PX'|=\lambda \tanh |PX|$ ($0<\lambda \leq 1$) is a collinearity which maps to a line to another line. In fact, if the line $l$ given for which the point $P$ doesn't belong to $l$, then we can consider that point $T$ of it for which $PT$ is perpendicular to $l$. Let another point of the line $X$, and consider the triangles $PTX_\triangle$ and $PT'X'_\triangle$, respectively. Let denote the angle $TPX_{\angle}$ by $\alpha$. Then $\cos\alpha=\tanh PT/\tanh PX$ since the first triangle is a right one. But $\tanh PT/\tanh PX=\tanh PT'/\tanh PX'$ so the second triangle is a right triangle. Hence the points $X'$ are on the line through the point $T'$ and orthogonal to the line $PT$. Hence this mapping takes a convex set into a convex set. Consider three points $A,B,C$ on the boundary of $K$ and the sets $K_A$, $K_B$ and $K_C$, respectively. We prove that $K_A\cap K_B\cap K_C$ is non-empty. Consider the median point $M$ of the triangle $ABC_\triangle$. (If $M_A,M_B$ and $M_C$ are the middle point of the sides $BC$, $AC$ and $AB$, respectively, then (by Menelaos theorem) the medians $MM_A$, $MM_B$ and $MM_C$ intersect each other in a point $M$.) Since the median point $M$ holds the equality
$$
\frac{\sinh AM}{\sinh MM_A}=2\cosh \frac{a}{2}=\frac{\cosh BM+\cosh CM}{\cosh MM_A}
$$
(see Theorem 4.3 in \cite{ghorvath_1}, or Theorem 1.2 in \cite{ghorvath_2} and the manuscript \cite{ghorvath_3} containing the proofs of these statements), we get
$$
\frac{\tanh AM}{\tanh MM_A}=\frac{\sinh AM\cosh MM_A}{\sinh MM_A\cosh AM}=\frac{\cosh BM+\cosh CM}{\cosh AM}
$$
On the other hand we have that
$$
\frac{\tanh AM}{\tanh AM_A}=\frac{\tanh AM+\tanh^2 AM\tanh MM_A}{\tanh AM+\tanh MM_A}=\frac{\frac{\tanh AM}{\tanh MM_A}+\tanh^2 AM}{\frac{\tanh AM}{\tanh MM_A}+1}=1-\frac{1-\tanh^2 AM}{\frac{\tanh AM}{\tanh MM_A}+1}=
$$
$$
=1-\frac{1}{\cosh ^2AM\left(\frac{\tanh AM}{\tanh MM_A}+1\right)}=1-\frac{1}{\cosh AM\left(\cosh AM+\cosh BM +\cosh CM\right)}\leq 1-\frac{1}{3\cosh^2d}.
$$
Since Helly's theorem on convex, compact sets of the hyperbolic plane is also valid we get that there is a point $O$ in the body $K$ that for each chord $PP'$ through $O$ has the property $\tanh PO\leq (1-\frac{1}{3\cosh^2d})\tanh PP'$.

We prove that the circle with center $O$ and radius $r$ in Eq. \ref{eq:blaschke} lies in $K$. Let $P$ be a boundary point of $K$. Since $K$ is a h-convex compact body there is a strip symmetric to its leading line such that one of its bounding hypercycle tangents $K$ at $P$ (see Statement \ref{st:isometricdiam} and Corollary \ref{cor:equivalence}). The other side of the strip supports $K$ at $Q$. Now consider the chord $QR$ which goes through $O$, and a line $l$ through $O$ which perpendicular to the generator $YX$ of the strip.
\begin{figure}[ht]
\includegraphics[scale=0.6]{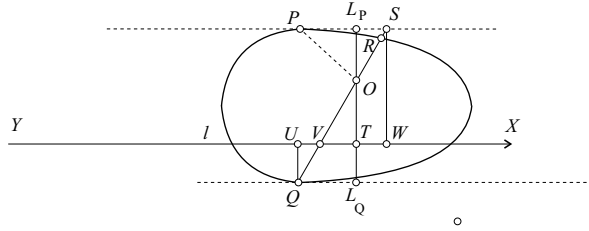}
\caption{The proof of Theorem \ref{thm:blaschke}.}\label{fig:blaschke}
\end{figure}
$l$ intersects the bounding hypercycles and the line $YX$ at the points $L_P$,$L_Q$ and $T$ respectively (see Fig. \ref{fig:blaschke}). $S$ and $V$ the intersection of $QR$ with the hypercycle (is going through $P$) and the line $YX$, respectively. The orthogonal projections of the points $Q$ and $S$ to the line $YX$ are $U$ and $W$. Using the hyperbolic theorem of sines we get
$$
\sin OVT_{\angle}=\frac{\sinh QU}{\sinh QV}=\frac{\sinh OT}{\sinh OV}=\frac{\sinh SW}{\sinh SV}
$$
Since $L_QO=OT+UQ$ and $L_QL_P=SW+UQ$ we get
$$
\frac{\tanh L_QO}{\tanh L_QL_P}= \frac{\tanh (OT+UQ)}{\tanh (SW+UQ)}=\frac{(\tanh OT+\tanh UQ)(1+\tanh SW\tanh UQ)}{(\tanh SW+\tanh UQ)(1+\tanh OT\tanh UQ)}=
$$
$$
=\frac{(\sinh OT\cosh UQ+\sinh UQ\cosh OT)(\cosh SW\cosh UQ+\sinh SW\sinh UQ)}{(\sinh SW\cosh UQ+\sinh UQ\cosh SW)(\cosh OT\cosh UQ+\sinh OT\sinh UQ)}=
$$
$$
=\frac{(\sinh OV\cosh UQ+\sinh QV\cosh OT)(\cosh SW\cosh UQ+\sinh SV\sinh QV)}{(\sinh SV\cosh UQ+\sinh QV\cosh SW)(\cosh OT\cosh UQ+\sinh OV\sinh QV)}\leq
$$
$$
\frac{\max\{\cosh UV,\cosh VT, \cosh VW\}}{\min\{\cosh UV,\cosh VT, \cosh VW\}}\frac{\tanh QO}{\tanh QS}\leq \frac{\max\{\cosh UV, \cosh VW\}}{\min\{\cosh UV,\cosh VW\}}\frac{\tanh QO}{\tanh QR}=\frac{\tanh QO}{\tanh QR}.
$$
(The last equality follows from the fact, that the strip symmetric to its leading line, hence $UV=VW$.)
Hence we get
$$
\tanh L_QO\leq \frac{\tanh QO}{\tanh QR}\tanh L_QL_P\leq c\tanh L_QL_P,
$$
where
$$
c:=\min\frac{\tanh QO}{\tanh QR}=\left(1-\frac{1}{3\cosh^2d}\right).
$$
This implies
$$
\tanh OP\geq \tanh OL_P=\tanh (L_QL_P-L_QO)=\frac{\tanh L_QL_P-\tanh L_QO}{1-\tanh L_QL_P\tanh L_QO}\geq \frac{(1-c)\tanh L_QL_P}{1-c\tanh^2L_QL_P}\geq
$$
$$
\geq \frac{(1-c)\tanh t}{1-c\tanh^2 t}=\frac{\frac{1}{3\cosh^2d}}{1-(1-\frac{1}{3\cosh^2d})\tanh^2 t}\tanh t=\frac{\tanh t}{3\cosh^2d(1-\tanh^2 t)+\tanh^2 t}=$$
$$
\frac{\tanh t\cosh^2t}{3\cosh^2d+\sinh^2 t}=\frac{\sinh t\cosh t}{3\cosh^2d+\sinh^2 t},
$$
because $\tanh (x-y)$ is a strictly decreasing function of $y$ and $f(x)=\frac{(1-c)\tanh x}{1-c\tanh^2 x}$ is a strictly increasing functions of $x$. Since the distance of an arbitrary boundary point from $O$ is at least the required value, the statement is proved.
\end{proof}

\end{document}